\documentclass{amsart}
\usepackage{amsmath, amsthm, amscd, amsfonts, amssymb, graphicx, color,enumitem}
\usepackage{pgf,tikz}
\usetikzlibrary{arrows}
\usepackage{float}
\newtheorem{theorem}{Theorem}[section]

\newtheorem{corollary}[theorem]{Corollary}

\newtheorem{notation}{Notation}
\theoremstyle{definition}

\newtheorem{definition}[theorem]{Definition}
\newtheorem{example}[theorem]{Example}
\newtheorem{observation}[theorem]{Observation}
\def\td{\gamma_3}
\newcounter{fig}
\newcommand{\newfig}[1]{
\refstepcounter{fig} \label{#1}%
Figure
\thefig
}

\begin{document}

\title{three  domination number and connectivity in graphs}
\author{S. Mehry and R. Safakish}
\maketitle
%%%%%%%%%%%%%%%%%%%%
%%%%%%%%%%%%%%%%%%%%%%%%%%%

\begin{abstract}
In a graph $G$, a vertex dominates itself
and its neighbors. A subset $S$ of $V$ is called a
dominating set in $G$
 if every vertex in $V$ is
dominated by at least one vertex in $S$. The
domination number $\gamma(G)$ is the minimum
cardinality of a dominating set. A set $S\subseteq V$ is
called a double dominating set of a graph $G$ if
every vertex in $V$ is dominated by at least two
vertices in $S$. The minimum cardinality of a
double dominating set is called double
domination number of $G$. The connectivity $\gamma(G)$ of a connected
graph $G$ is the minimum number of vertices
whose removal results in a disconnected or
trivial graph.  In this paper,
introduced the
concept of three  domination in graphs.
and we obtain an upper
bound for the sum of the three
domination number and connectivity of a
graph and characterize the corresponding
extremal graphs.
\end{abstract}
\vskip 0.2 true cm

\section{Introduction}
By the graph $G = (V,E)$ we mean a
finite, undirected and connected graph with
neither loops nor multiple edges.% The order
%and size of $G$ are denoted by $m$ and $n$ respectively.
 The {\bf degree} of any vertex $u$ in
$G$ is the number of edges incident with $u$
and is denoted by $d(u)$. The minimum and
maximum degree of a graph $G$ is denoted
by $\Delta(G)$ and $\delta(G)$ (or $\Delta$ and $\delta$) respectively.
The open neighbourhood of a vertex $v\in V$, denoted by $N(v)$ is the set of all
 vertices adjacent to $v$.
 A {\bf vertex (edge) cut} , or {\bf cut-vertex (cut-edge)} of a graph $G$ is a vertex whose removal increases the number of components. The
{\bf connectivity} $\kappa(G)$ of a connected graph $G$
is the minimum number of vertices whose
removal results in a disconnected or trivial
graph. 
The  {\bf union} of the two graphs $G$ and $H$, written as $G\cup H$  will have vertex set   $V(G)\cup V(H)$ 
and edge set $E(G)\cup E(H)$.
The  {\bf sum}  of two graphs $G$ and $H$, written as $G+ H$, is obtained by first forming the union $G\cup H$ and then, making every vertex of $G$ adjacent to every vertex of $H$. 
A {\bf bipartite} graph is graph whose vertex set can be divided into disjoint set $V_1$ and $V_2$
 such that every edge has one end in $V_1$ and another end in $V_2$. A  {\bf complete bipartite}
 graph is a bipartite graph where every vertex of $V_1$ is adjacent to every vertex in $V_2$.
 The   complete bipartite
 graph with  partitions of order
  $|V_1|=m$ and  $|V_2|=n$, is denoted by $K_{m,n}$.
     A {\bf  friendship}  graph, denoted by $F_n$
  can be constructed by identifing $n$ copies of the cycle $C_3$ at common vertex. A {\bf wheel} graph, denoted by $W_n$ is a graph with $n$ vertices, formed by connecting a single vertex to all vertices  of an (n-1) cycle.
%%%%%%%%%%%%%%%%%%%%%%%%%%%
For graph
theoretic terminology we refer to Chartrand
and Lesniak \cite{a1} and Haynes et .al \cite{a2,a3}.
In a graph $G$, a vertex dominates
itself and its neighbors. A subset $S$ of $V$ is
called a dominating set in $G$ if every vertex
in $V$ is dominated by at least one vertex in
$S$. The domination number $\gamma(G)$ is the
minimum cardinality of a dominating set.
Harary and Haynes \cite{a4} introduced the
concept of double domination in graphs. A
set $S\subseteq V$ is called a double dominating set
of a graph $G$ if every vertex in $V$  is
dominated by atleast two vertices in $S$. The
minimum cardinality of double dominating
set is called double domination number of
$G$. 
Several authors have studied the
problem of obtaining an upper bound for
the sum of a domination parameter and a
graph theoretic parameter and
characterized the corresponding extremal
graphs. J.Paulraj Joseph and S.Arumugam
\cite{a5} proved that $\gamma(G)+\kappa (G)\leqslant n$ and
characterized the corresponding extremal
graphs.
We introduce the
concept of three  domination in graphs and  obtain an upper
bound for the sum of the three
domination number and connectivity of a
graph and characterize the corresponding
extremal graphs.
\section{preliminaries}
\begin{notation}
Let $G$ be a connected graph with $m$ vertices $v_1,v_2,\ldots,v_m$. The graph obtained from $G$  by attaching $n_i$ times a pndant vertex of $P_{l_i}$ on vertex $v_i$  is denoted by $G(n_1p_{l_1},n_2p_{l_2},\ldots n_mp_{l_m})$ where
$n_i,l_i\geqslant 0$ and $1\leqslant i\leqslant m$.
\end{notation}
\begin{example}
Let $v_1, v_2, v_3, v_4$ be the vertices of $C_4$. The graph $C_4(P_2, 2P_3, P_4, P_3)$
 is obtained from $C_4$ by attaching $1$ time
a pendant vertex of $P_2$ on $v_1$, $2$ time a pendant vertex of $P_3$ on $v_2$, $1$ time a pendant vertex of $P_4$ on $v_3$
 and $1$ time a pendant vertex of $P_3$ on $v_4$ as is shown in Figure \ref{ex1}
\end{example}
%\begin{figure}[H]
\begin{center}
\begin{tikzpicture}[line width=.4mm,line cap=round,line join=round,>=triangle 45,x=1.50cm,y=1.0cm]
\draw (1.,1.)-- (2.,1.);
\draw (1.,1.)-- (1.,0.);
\draw (1.,0.)-- (2.,0.);
\draw (2.,0.)-- (2.,1.);
\draw (1.,0.)-- (0.,0.);
\draw (1.,1.)-- (0.5,1.5);
\draw (0.,1.5)-- (0.5,1.5);
\draw (1.,1.)-- (0.5,1.);
\draw (0.5,1.)-- (0.,1.);
\draw (2.,1.)-- (2.5,1.);
\draw (2.5,1.)-- (3.,1.);
\draw (3.,1.)-- (3.5,1.);
\draw (2.,0.)-- (2.5,0.);
\draw (2.5,0.)-- (3.,0.);
\draw [fill=black] (1.,1.) circle (2pt);
\draw [fill=black] (2.,1.) circle (2pt);
\draw [fill=black] (1.,0.) circle (2pt);
\draw [fill=black] (2.,0.) circle (2pt);
\draw [fill=black] (0.,0.) circle (2pt);
\draw [fill=black] (0.5,1.5) circle (2pt);
\draw [fill=black] (0.,1.5) circle (2pt);
\draw [fill=black] (0.5,1.) circle (2pt);
\draw [fill=black] (0.,1.) circle (2pt);
\draw [fill=black] (2.5,1.) circle (2pt);
\draw [fill=black] (3.,1.) circle (2pt);
\draw [fill=black] (3.5,1.) circle (2pt);
\draw [fill=black] (2.5,0.) circle (2pt);
\draw [fill=black] (3.,0.) circle (2pt);
\begin{scriptsize}
\draw (0.85,-0.15) node[anchor=north west] {$v1$};
\draw (0.85,1.5) node[anchor=north west] {$v_2$};
\draw (1.85,1.5) node[anchor=north west] {$v_3$};
\draw (1.85,-0.15) node[anchor=north west] {$v_4$};
\draw (.8,-0.7) node[anchor=north west] {$C_4(P_2, 2P_3, P_4, P_3)$};
\end{scriptsize}
\end{tikzpicture}\\
\newfig{ex1}
\end{center}
%\caption{}\label{ex1}
%\end{figure}
\begin{definition}
A set $S\subseteq V$ is called a three dominating set
of a graph $G$ if every vertex in $V\setminus S$ is dominated by   atleast three vertices in $S$. The
minimum cardinality of three dominating
set is called three  domination number of
$G$ and is denoted by $\td(G)$.
\end{definition}
\begin{example}
$\td(K_n)=3$, for $n\geqslant 3$ and $\td(P_n)=\td(C_n)=n$. For graphs $G_1$ and $G_2$ in Figure \ref{ex24},
$S_1=\{v_1,v_3,v_5\}$  and  $S_2=\{v_3,v_4,v_5,v_6\}$ are a three domination set respectively. Then, $\td(G_1)=3$ and $\td(G_2)=4$.
\begin{center}
\begin{tikzpicture}[line width=.4mm,line cap=round,line join=round,>=triangle 45,x=1.0cm,y=1.0cm]
\coordinate (a2) at (2.796721281221858,1.628968706251709);
\coordinate  (a1) at  (2.,3.);
\coordinate (a6) at  (0.4142914296196061,2.995534777600331) ;
\coordinate  (a5) at  (-0.37469585953893025,1.6200382614523703);
\coordinate (a4) at (0.4220254216829272,0.24900696770407937);
\coordinate (a3) at (2.007733992063321,0.253472190103748);
\path (a1) ++(6,0) coordinate (b1);
\path (a2) ++(6,0) coordinate (b2);
\path (a3) ++(6,0) coordinate (b3);
\path (a4) ++(6,0) coordinate (b4);
\path (a5) ++(6,0) coordinate (b5);
\path (a6) ++(6,0) coordinate (b6);
\draw (b1) --(b2) ;
\draw (b1) --(b3) ;
\draw (b1) --(b4) ;
\draw (b6) --(b3) ;
\draw (b6)  -- (b4) ;
\draw (b6) --(b5) ;
\draw (b5) --(b2) ;
%
%\begin{scriptsize}
\draw [fill=black] (b1)  circle  (2pt)   ;
\draw [fill=black] (b2)  circle  (2pt)   ;
\draw [fill=black] (b6)  circle  (2pt)  ;
\draw [fill=black] (b4)  circle  (2pt)  ; 
\draw [fill=black] (b5)  circle  (2pt)  ;
\draw [fill=black] (b3)  circle  (2pt)  ;
\draw [fill=black] (a1)  circle  (2pt)   ;
\draw [fill=black] (a2)  circle  (2pt)   ;
\draw [fill=black] (a6)  circle  (2pt)  ;
\draw [fill=black] (a4)  circle  (2pt)  ; 
\draw [fill=black] (a5)  circle  (2pt)  ;
\draw [fill=black] (a3)  circle  (2pt)  ;
\begin{scriptsize}
\draw (0.2,3.5) node[anchor=north west] {$v_1$};
\draw (1.8,3.5) node[anchor=north west] {$v_2$};
\draw (2.9,1.8) node[anchor=north west] {$v_3$};
\draw (2.,0.1) node[anchor=north west] {$v_4$};
\draw (0.1,0.1) node[anchor=north west] {$v_5$};
\draw (-0.95,1.8) node[anchor=north west] {$v_6$};
\draw (6.2,3.5) node[anchor=north west] {$v_1$};
\draw (7.8,3.5) node[anchor=north west] {$v_2$};
\draw (8.9,1.8) node[anchor=north west] {$v_3$};
\draw (8.,0.1) node[anchor=north west] {$v_4$};
\draw (6.1,0.1) node[anchor=north west] {$v_5$};
\draw (5.05,1.8) node[anchor=north west] {$v_6$};
\draw (1,-.3) node[anchor=north west] {$G_1$};
\draw (7,-.3) node[anchor=north west] {$G_2$};
\end{scriptsize}
\draw (a1)--(a2)--(a3) --(a4) --(a5)--(a6)--(a1) ;
\draw (a1) --(a4) ;
\draw (a2) --(a5) ;
\draw (a6) --(a3) ;
\end{tikzpicture}\\
\newfig{ex24}
\end{center}
\end{example}
\begin{observation} 
 $3\leqslant \td(G)\leqslant n$. (For $n\geqslant 3$)
\end{observation}
\begin{theorem}\label{t25}
For any graph $G$, $\td (G)= n$ if and only if $\Delta(G)\le2$.
\end{theorem}
\begin{proof}
Suppose  $v\in V$. Then, $d(v)\ge 3$ if and only if  $V\setminus\{v\}$ be  a three dominating set. 
\end{proof}
\begin{theorem}\cite{a1}
For any graph $G$, $\kappa(G)\leqslant\delta(G)$.
\end{theorem}
\begin{theorem}\label{t27}
Suppose $n\geqslant 5$,  $M$ is an matching of $K_n$ and $G=K_n\setminus\! M$. If
 $M$ is a perfect matching, then, $\td(G)=4$ and otherwise $\td(G)=3$.
 \end{theorem}
 \begin{proof}
Suppose $M$ is not  a perfect matching. Then, $V(G)\setminus\! V(M)\neq\varnothing$.
 If  $x\in V (G)\setminus\! V(M)$ and $uv\in E(M)$  then, $\{x,u,v\}$ is a three dominating set in $G$.  Therefore $\td(G)=3$.
 Now your assumption M is  a perfect matching  and $U=\{v_1,v_2,v_3\}\subset V(G)$.
 In this case, since $M$ is a perfect, there is a vertex $w\in V(G)\setminus\! U$ which is not adjacent to one of the vertices of $U$. Hence $\td(G)>3$.
 It is clear that any four-element subset of $V$ is a three dominating  set. Hence $\td(G)=4$.
 \end{proof}
 \begin{corollary}
  If is $n$  an odd intager,  then, $\td(K_n\setminus\! M)=3$ where $M$ is  a matching.
 \end{corollary}
%%%%%%%%%%%%%%%%%%%%%%%%%%%
%%%%%%%%%%%%%%%%%%%%%%%%%%%
\section{main results}
\begin{theorem}
For any  graph $G$,
$\td(G)+\kappa(G)\leqslant  2n-1$ and equality holds
if and only if $G$ is isomorphic to $K_3$.
\end{theorem}
\begin{proof}
$\td(G)+\kappa(G)\leqslant n+\delta(G)\leqslant  2n-1$.
Let $\td(G)+\kappa(G)=2n-1$. Then, $\td(G)=n$ and $\kappa(G)=n-1$. Then, $G$
 is a complete graph on $n$ vertices. Since $\td(K_n)=3$ we
have $n=3$. Hence $G$ is isomorphic to $K_3$.
The converse is obvious.
\end{proof}
%%%%%%%%%%%%%%%%%%%%%%%%%
%%%%%%%%%%%%%%%%%%%%%%%%
\begin{theorem}\label{t32}
For any connected graph
$G$, $\td(G)+\kappa (G)=2n-2$ if and only if $G$
is isomorphic to $K_4$, $C_4$ or $K_{1,2}$.
\end{theorem}\begin{proof}
 Let $\td(G)+\kappa(G)=2n-2$. Then,
there are two cases to consider.
\begin{enumerate}
\item[\bf i.] $\td(G)=n-1$ and $\kappa(G)=n-1$.
\item[\bf i\!\! i.] $\td(G)=n$ and $\kappa(G)=n-2$.
\end{enumerate}
{\bf Case i}. $\td(G)=n-1$ and $\kappa(G)=n-1$.
Then, $G$ is a complete graph on $n$
vertices. Since $\td(K_n)=3$ we have $n-1=3$ and $n=4$.
Hence $G$ is isomorphic to $K_4$.\\
{\bf Case i\!\! i}. $\td(G)=n$ and $\kappa(G)=n-2$.
Then, $n-2\leqslant \delta (G)$.
 If $\delta (G)=n-1$, then,
$G$ is a complete graph which is a
contradiction. 
Hence $\delta (G)=n-2$. Then, $G$ is
isomorphic to $K_n-M$ where $M$ is a matching
in $K_n$. Then, $n=\td(K_n-M)=3$ or $4$.  For  
 $n=4$,  $G$ is  isomorphic to $C_4$ and for  $n=3$, $G\simeq K_{1,2}$. The converse is
obvious.
\end{proof}
%%%%%%%%%%%%%%%%%%%%%%%%%%%%%%%
%%%%%%%%%%%%%%%%%%%%%%%
\begin{theorem}\label{t33}
For any  connected graph $G$
$\td(G)+\kappa(G)=2n-3$ if and only if $G$ is
isomorphic to one of the graphs $K_5$,  $C_5$,  $P_4$.%
%, $K_4-e$,  $K_1\cup K_2$ or $\overline{K_3}$.($M$ is a matching and  $e$ is a edge in  $G$.)
\end{theorem}
\begin{proof}
Let
$\td(G)+\kappa(G)=2n-3$. Then,
there are three cases to consider
\begin{enumerate}
\item[\bf i.] $\td(G)=n-2$ and $\kappa(G)=n-1$.
\item[\bf i\!\! i.] $\td(G)=n-1$ and $\kappa(G)=n-2$.
\item[\bf i\!\! i\!\! i.]$\td(G)=n$ and $\kappa(G)=n-3$.
\end{enumerate}
{\bf Case i}. 
In this case, like  Theorem \ref{t32}  $G$ is isomorphic to $K_5$.\\
{\bf Case i\!\! i}.
 $\td(G)=n-1$ and $\kappa(G)=n-2$.
Then, $n-2\leqslant \delta (G)$.
% If $\delta (G)=n-1$ then $G$ is a complete graph which is acontradiction. 
Hence $\delta (G)=n-2$ and  $G$ is
isomorphic to $K_n-M$ where $M$ is a matching in $K_n$.
 Hence, from Theorem \ref{t27}, $\td(G)=3$ or $4$  and 
both cases are contradictory under the existing the conditions of theorem.
% هر دو حالت با شرایط قضیه متناقض است
 \\
 {\bf Case i\!\! i\!\! i}. $\td(G)=n$
 and $\kappa(G)=n-3$.
Then, $n-3\leqslant \delta (G)$. In this case $G\simeq C_n$ or  $G\simeq P_n$ (From Theorem \ref{t25}).
We have $\kappa(C_n)=2$ and $\kappa(P_n)=1$. Hence $G$ is isomorphic to $C_5$ or $P_4$. The converse is obvious.
\end{proof}
%%%%%%%%%%%%%%%%%%%%%%%%%%%%
%%%%%%%%%%%%%%%%%%%%%%%%%%%%%%%%%%%%%
\begin{theorem}\label{t34}
For any connected  graph $G$,
$\td(G)+\kappa(G)=2n-4$ if and only if $G$ is
isomorphic to  one of the graphs
   $K_6$, $K_6\setminus\! M$ (where $M$ is a perfect matching.), $C_6$,   $K_5\setminus\! M$ (where $M$ is a matching.), $P_5$,  $P_4$,  $C_3(P_2,0,0)$, $K_{1,3}$, $K_1+P_4$,
or  graph obtained from the connecting two non-adjacent vertices of $C_5$.
   % $\overline{K_4}$, $K_2\cup K_2$, $K_2\cup K_1\cup K_1$,   or graph $H_1$ as shown in Figure \ref{ft9}
\end{theorem}
\begin{proof}
Take
$\td(G)+\kappa(G)=2n-4$. Then,
there are four  cases to consider
\begin{enumerate}
\item[\bf i.] $\td(G)=n-3$ and $\kappa(G)=n-1$.
\item[\bf i\!\! i.] $\td(G)=n-2$ and $\kappa(G)=n-2$.
\item[\bf i\!\! i\!\! i.] $\td(G)=n-1$ and $\kappa(G)=n-3$.
\item[\bf i\!\! v.] $\td(G)=n$ and $\kappa(G)=n-4$.
\end{enumerate}
{\bf Case i}. In this case $G$ is isomorphic to $K_6$.\\
{\bf Case i\!\! i}. $\td(G)=n-2$ and $\kappa(G)=n-2$.
Then, $n-2\leqslant \delta (G)$.
 If $\delta (G)=n-1$, then,
$G$ is a complete graph which is a
contradiction. 
Hence $\delta (G)=n-2$. Then, $G$ is
isomorphic to $K_n\setminus\! M$ where $M$ is a matching
in $K_n$. Then, $\td(G)=3$ or $4$.
%%مثلا اگر  n زوج باشد و یک تطابق کامل حذف کنیم dd(G)=4
For $\td(G)=3$, $G$
isomorphic to $K_5\setminus\! M$ and  
from Theorem \ref{t27}, for $\td(G)=4$, $G$ is isomrphic to $K_6\setminus\! M$ where $M$ is a perfect mathching.\\
{\bf Case i\!\! i\!\! i}. $\td(G)=n-1$
 and $\kappa(G)=n-3$.
Then, $n-3\leqslant \delta (G)\leqslant n-2$. 
$\delta (G)=n-2$ is inconsistent with the assumptions of Theorem.
 Hence  $\delta(G)=\kappa(G)=n-3$. Suppose 
 $U=\{u_1,u_2,\ldots,u_{n-3}\}$ br vertex cut of $G$ and let 
$V\setminus\! U=\{v_1,v_2,v_3\}$. There are two  cases for $\langle V\setminus\! U\rangle$ to consider
\begin{enumerate}
\item[\bf i\!\! i\!\! i.a] $\langle V\setminus\! U\rangle=\overline{K_3}$.
\item[\bf i\!\! i\!\! i.b] $\td(G)=K_1\cup K_2$.
\end{enumerate}
%%%%%%%%%%%%%%%%%%%%%%%%%%%%%%%%%%%%%%%
{\bf Case i\!\! i\!\! i.a}. In this  case,  every vertex of $V\setminus\! U$ is
adjacent to all the vertices of $U$. Then, $V\setminus\! U$ is a three dominating set of $G$ 
hence, $n-1=\td(G)=3$ and $n=4$. Hence  $G\simeq K_{1,3}$.\\
{\bf Case i\!\! i\!\! i.b.}
 Suppos $V( K_2)=\{v_1,v_2\}$. Then, $v_3$ is adjacent to all the vertices in $U$ and
 $v_1$, $v_2$ are not adjacent  to at most one vertex in $U$.
 Suppose exsist $x,y\in U$  (It is possible that  $x=y$) so that
  $xv_1,yv_2\not\in E(G)$. Then, $\{x,y,v_1,v_2,v_3\}$ is a three dominating set.
 Then, $n-1=\td(G)\le 5$. This gives $4\leqslant n\leqslant 6$.
 $n=6$ is impossible.
 because $v_2\in N(v_1)$ and  at least two vertices of of $U$ belong to  $N(v_1)$. Henc $V\setminus\{v_1,v_3\}$ is a three dominating set of $G$ which is a contradiction.
  For $n=5$, $G$ is isomorphic to  $K_1+P_4$ or
  or the graph obtained from connecting two non-adjacent vertices of $C_5$. 
 If $n=4$, then, G is isomorphic to $P_4$ or  $C_3(P_2,0,0)$.\\
{\bf Case i\!\! v}. $\td(G)=n$
 and $\kappa(G)=n-4$. 
 From theorem \ref{t25},  $d(v)\leqslant 2$ for all $v\in V$. hence,
$$n-4=\kappa(G)\leqslant\delta(G)\leqslant 2\Rightarrow 4\leqslant n\leqslant6$$
$n=4$, is impossible then, for $n=6$ and $n=5$, $G\simeq C_6$ and $G\simeq P_5$, respectively. The converse is obvious.
\end{proof}
%%%%%%%%%%%%%%%%%%%%%%%%%%%%%%
%%%%%%%%%%%%%%%%%%%%%%%%%%
\begin{theorem}
For any connected graph G
$\td(G)+\kappa(G)=2n-5$ if and only if $G$ is
isomorphic to  one of the graphs 
$K_7$, $K_6-M$ ($M$ is a matching with $|M|<3$.), 
$\overline{P_2\cup P_2}$, $\overline{C_4\cup K_1\cup K_1}$,
$\overline{C_4\cup P_2}$, $\overline{P_6}$, $\overline{C_6}$, $C_6$, $P_6$, $W_5$,
$K_{2,3}$,$K_{1,4}$, $K_2+K_3$, $\overline{P_3\cup P_2}$, $\overline{P_3\cup K_1\cup K_1}$,
$F_2$, $C_4(P_2,0,0,0)$,$P_3(0,P_3,0)$, $C_3(2P_2,0,0)$, $C_3(P_2,P_2,0)$
  or $T_i$ ($7\le i\le 12$) as shown in Figures 
\ref{ft104},\ref{ft105},\ref{ft107}
\end{theorem}
\begin{proof}
Let
$\td(G)+\kappa(G)=2n-5$. Then,
there are f five cases to consider
\begin{enumerate}
\item[\bf i.] $\td(G)=n-4$ and $\kappa(G)=n-1$.
\item[\bf i\!\! i.] $\td(G)=n-3$ and $\kappa(G)=n-2$.
\item[\bf i\!\! i\!\! i.] $\td(G)=n-2$ and $\kappa(G)=n-3$.
\item[\bf i\!\! v.] $\td(G)=n-1$ and $\kappa(G)=n-4$.
\item[\bf  v.] $\td(G)=n$ and $\kappa(G)=n-5$.
\end{enumerate}
{\bf Case i.} $G\simeq K_7$.\\
{\bf Case i\!\! i.} In this case $\delta (G)=n-2$. Hence, $G$ is
isomorphic to $K_n\setminus\! M$ where $M$ is a matching
in $K_n$. Then, $\td(G)=3$ or $4$.
%%مثلا اگر  n زوج باشد و یک تطابق کامل حذف کنیم dd(G)=4
 Since $\td(G)=4$ is impossible (From Theorem\ref{t27}), we have $\td(G)=3$.
  Hence, $G$
isomorphic to $K_6\setminus\! M$ where $M$ is a  matching in $K_6$ with $|M|<3$.\\
{\bf Case i\!\! i\!\! i}. $\td(G)=n-2$
 and $\kappa(G)=n-3$.
Then, $n-3\leqslant \delta (G)\leqslant n-2$ or  $\delta (G)=n-2$.  
 Then, $n-2=\td(G)=3$ or $4$.
  $\td(G)=3$ is impossible (because $\kappa(K_5\setminus\! M)=3\neq 2$).
 Hence,  we have  $\td(G)=4$ and $n=6$.  Hence,
$G$ is isomorphic to  $K_6\setminus\! M$ which is a contradiction with $\kappa(G)$.
Let $\delta (G)=n-3=\kappa(G)$ and $U=\{u_1,u_2,\ldots,u_{n-3}\}$ be vertex cut of $G$ and let 
$V\setminus\! U=\{v_1,v_2,v_3\}$.
%%%%%%%%%%%%%%%%%%%%%%%%%%%%%%%%%%%%%%%
If $\langle V\setminus\! U\rangle=\overline{K_3}$,
Then, every vertex of $V\setminus\! U$ is
adjacent to all the vertices of $U$. Then, $V\setminus\! U$ is a three dominating set of $G$ 
hence, $n-2=\td(G)=3$ and $n=5$.
In this case  $G$ is  isomorphic to $K_{2,3}$, $K_2+K_3$.

%%%%%%%%%%%%%%%%%%%%%%%%%%%%%%%
Let $\langle V\setminus\! U\rangle=K_1\cup K_2$.
As in the  Theorem \ref{t34}{\bf.i\!\! i\!\! i.b.} we have
$\td(G)\leqslant 5$ and  $5\leqslant n\leqslant 7$.  Cases $n=3,4$ are impossible.
For $n=7$, we have $\delta(G)=4$. Hence
each  vertex  $V\setminus\! U$ is adjacent to all vertices in $U$. So $U$ is a three dominating set of $G$. then, $\td(G)\le|U|=n-3$ which is a contradiction.  For $n=5$, $G$ is isomorphic to one of graphs $H_1$ or $H_2$ as shown in Figure \ref{ft101}
\begin{center}
\begin{tikzpicture}[line width=.4mm,line cap=round,line join=round,>=triangle 45,x=1.0cm,y=1.0cm]
\draw (0.,1.)-- (0.5,0.);
\draw (0.,1.)-- (1.,2.);
\draw (1.5,0.)-- (0.5,0.);
\draw (0.,1.)-- (2.,1.);
\draw (2.,1.)-- (1.,2.);
\draw (0.,1.)-- (1.5,0.);
\draw (1.5,0.)-- (2.,1.);
\begin{scriptsize}
\draw (-0.3,-0.4) node[anchor=north west] {$H_1=\overline{P_3\cup K_1\cup K_1}$};
\draw (3.5,-0.4) node[anchor=north west] {$H_2=\overline{P_3\cup P_2}$};
\end{scriptsize}
\draw (2.,1.)-- (0.5,0.);
\draw (3.5,1.)-- (4.,0.);
\draw (4.,0.)-- (5.,0.);
\draw (5.,0.)-- (5.5,1.);
\draw (5.5,1.)-- (4.5,2.);
\draw (4.5,2.)-- (3.5,1.);
\draw (3.5,1.)-- (5.,0.);
\draw (5.5,1.)-- (4.,0.);
\draw [fill=black] (0.,1.) circle (2pt);
\draw [fill=black] (0.5,0.) circle (2pt);
\draw [fill=black] (1.,2.) circle (2.0pt);
\draw [fill=black] (1.5,0.) circle (2pt);
\draw [fill=black] (2.,1.) circle (2.0pt);
\draw [fill=black] (3.5,1.) circle (2.0pt);
\draw [fill=black] (4.,0.) circle (2.0pt);
\draw [fill=black] (5.,0.) circle (2.0pt);
\draw [fill=black] (5.5,1.) circle (2.0pt);
\draw [fill=black] (4.5,2.) circle (2.0pt);
\end{tikzpicture}\\
\newfig{ft101}
\end{center}
 For $n=6$ $G$ is isomorphic to one of graphs $T_i$ , $1\leqslant i\leqslant 6$ in Figurr \ref{ft102}.
\begin{center}
\begin{tikzpicture}[line width=.4mm,line cap=round,line join=round,>=triangle 45,x=1.0cm,y=1.0cm]
\draw (0.,1.)-- (0.5,0.);
\draw (0.,1.)-- (1.,2.);
\draw (1.,2.)-- (2.,1.);
\draw (2.,1.)-- (1.5,0.);
\draw (1.5,0.)-- (0.5,0.);
\draw (0.,1.)-- (1.,1.);
\draw (1.,1.)-- (1.,2.);
\draw (1.,1.)-- (2.,1.);
\draw (2.,1.)-- (0.5,0.);
\draw (0.,1.)-- (1.5,0.);
\draw (1.5,0.)-- (1.,1.);
\draw (3.5,1.)-- (4.5,2.);
\draw (4.5,2.)-- (5.5,1.);
\draw (5.5,1.)-- (4.5,1.);
\draw (4.5,1.)-- (3.5,1.);
\draw (3.5,1.)-- (4.,0.);
\draw (4.,0.)-- (5.,0.);
\draw (5.,0.)-- (4.5,1.);
\draw (4.5,1.)-- (4.,0.);
\draw (3.5,1.)-- (5.,0.);
\draw (7.,1.)-- (8.,2.);
\draw (8.,2.)-- (8.,1.);
\draw (8.,2.)-- (9.,1.);
\draw (9.,1.)-- (7.,1.);
\draw (7.,1.)-- (7.5,0.);
\draw (9.,1.)-- (8.5,0.);
\draw (8.5,0.)-- (7.5,0.);
\draw (9.,1.)-- (7.5,0.);
\draw (7.,1.)-- (8.5,0.);
\draw (1.,-2.)-- (0.,-3.);
\draw (0.,-3.)-- (0.5,-4.);
\draw (0.5,-4.)-- (1.5,-4.);
\draw (1.5,-4.)-- (2.,-3.);
\draw (2.,-3.)-- (1.,-2.);
\draw (1.,-2.)-- (1.,-3.);
\draw (1.,-3.)-- (0.,-3.);
\draw (1.,-3.)-- (0.5,-4.);
\draw (0.,-3.)-- (1.5,-4.);
\draw (3.5,-3.)-- (4.5,-2.);
\draw (4.5,-2.)-- (4.5,-3.);
\draw (4.5,-3.)-- (5.5,-3.);
\draw (5.5,-3.)-- (4.5,-2.);
\draw (5.5,-3.)-- (5.,-4.);
\draw (5.,-4.)-- (4.,-4.);
\draw (4.,-4.)-- (4.5,-3.);
\draw (3.5,-3.)-- (4.,-4.);
\draw (3.5,-3.)-- (5.,-4.);
\draw (7.,-3.)-- (8.,-2.);
\draw (8.,-2.)-- (8.,-3.);
\draw (8.,-3.)-- (7.5,-4.);
\draw (7.5,-4.)-- (7.,-3.);
\draw (7.,-3.)-- (8.,-3.);
\draw (8.,-2.)-- (9.,-3.);
\draw (9.,-3.)-- (8.,-3.);
\draw (8.,-3.)-- (8.5,-4.);
\draw (8.5,-4.)-- (9.,-3.);
\draw (8.5,-4.)-- (7.5,-4.);
\draw [fill=black] (0.,1.) circle (2pt);
\draw [fill=black] (0.5,0.) circle (2pt);
\draw [fill=black] (1.,2.) circle (2.0pt);
\draw [fill=black] (2.,1.) circle (2.0pt);
\draw [fill=black] (1.5,0.) circle (2pt);
\draw [fill=black] (1.,1.) circle (2.0pt);
\draw [fill=black] (3.5,1.) circle (2.0pt);
\draw [fill=black] (4.5,2.) circle (2.0pt);
\draw [fill=black] (5.5,1.) circle (2.0pt);
\draw [fill=black] (4.5,1.) circle (2.0pt);
\draw [fill=black] (4.,0.) circle (2pt);
\draw [fill=black] (5.,0.) circle (2pt);
\draw [fill=black] (7.,1.) circle (2.0pt);
\draw [fill=black] (8.,2.) circle (2.0pt);
\draw [fill=black] (8.,1.) circle (2.0pt);
\draw [fill=black] (9.,1.) circle (2.0pt);
\draw [fill=black] (7.5,0.) circle (2pt);
\draw [fill=black] (8.5,0.) circle (2.0pt);
\draw [fill=black] (1.,-2.) circle (2.0pt);
\draw [fill=black] (0.,-3.) circle (2pt);
\draw [fill=black] (0.5,-4.) circle (2.0pt);
\draw [fill=black] (1.5,-4.) circle (2.0pt);
\draw [fill=black] (2.,-3.) circle (2.0pt);
\draw [fill=black] (1.,-3.) circle (2.0pt);
\draw [fill=black] (3.5,-3.) circle (2.0pt);
\draw [fill=black] (4.5,-2.) circle (2.0pt);
\draw [fill=black] (4.5,-3.) circle (2.0pt);
\draw [fill=black] (5.5,-3.) circle (2.0pt);
\draw [fill=black] (5.,-4.) circle (2.0pt);
\draw [fill=black] (4.,-4.) circle (2.0pt);
\draw [fill=black] (7.,-3.) circle (2.0pt);
\draw [fill=black] (8.,-2.) circle (2.0pt);
\draw [fill=black] (8.,-3.) circle (2.0pt);
\draw [fill=black] (7.5,-4.) circle (2.0pt);
\draw [fill=black] (9.,-3.) circle (2.0pt);
\draw [fill=black] (8.5,-4.) circle (2.0pt);
\begin{scriptsize}
\draw (0.1,-0.5) node[anchor=north west] {$T_1=\overline{P_2\cup P_2}$};
\draw (3.2,-0.5) node[anchor=north west] {$T_2=\overline{C_4\cup K_1\cup K_1}$};
\draw (7.1,-0.5) node[anchor=north west] {$T_3=\overline{C_4\cup P_2}$};
\draw (0.3,-4.4) node[anchor=north west] {$T_4=\overline{P_6}$};
\draw (3.8,-4.5) node[anchor=north west] {$T_5=\overline{C_6}$};
\draw (7.4,-4.5) node[anchor=north west] {$T_6=W_5$};
\end{scriptsize}
\end{tikzpicture}\\ \newfig{ft102}
\end{center}
%%%%%%%%%%%%%%%%%%%%%%%%%%%%%%%
{\bf Case i\!\! v}. $\td(G)=n-1$
 and $\kappa(G)=n-4$.
 Then $n-4\leqslant \delta (G)\leqslant n-2$. 
 $\delta (G)=n-2$ is impossible.
Let $\delta (G)=n-3$ and $U=\{u_1,u_2,\ldots,u_{n-4}\}$ be vertex cut of $G$ and take 
$V\setminus\! U=\{v_1,v_2,v_3,v_4\}$.
  If is $x$   an isolated vertex in $ \langle V\setminus\! U\rangle$,  then, $d(x)\leqslant n-4<\delta(G)$ which is a
contradiction. Hence, $\langle V\setminus\! U\rangle=K_2\cup K_2$.
Then every vertex of $V\setminus\! U$ is
adjacent to all the vertices of $U$ and  $V\setminus\! U$ is a three dominating set of $G$. 
hence, $\td(G)=3$ or $4$. Since $\td(G)=3$ is impossible, we have $\td(G)=4$ and $n=5$.
  In this case, $G$ is isomprphic to $F_2$.
%%%%%%%%%%%%%%%%%%%%%%%%%%%%%%%%%%%%%

Take $\delta (G)=n-4=\kappa(G)$ and $U=\{u_1,u_2,\ldots,u_{n-4}\}$ be vertex cut of $G$ and let 
$V\setminus\! U=\{v_1,v_2,v_3,v_4\}$.
If $\langle V\setminus\! U\rangle=\overline{K_4}$
Then every vertex of $V\setminus\! U$ is
adjacent to all the vertices of $U$ and  $V\setminus\! U$ is a three dominating set of $G$ 
hence, $\td(G)=3$ or $4$. Since $\td(G)=3$ is impossible, we have $\td(G)=4$ and $n=5$.
Hence  $G$ is to $K_{1,4}$.
Suppose $\langle V\setminus\! U\rangle=K_1\cup K_1\cup K_2$ and $v_1,v_2$ are isolated vertex $\in V\setminus\! U$. If  $n-4\geqslant 3$, then, $n\geqslant 7$ and  $V-\{v_1,v_2\}$ is a three dominating set which is a contradiction. Therefore $4\leqslant n\leqslant 6$. $n=4,6$ is impossible.
For $n=5$ $G$ is isomorphic to $P_3(0,P_3,0)$ or $C_3(2P_2,0,0)$. \\
 Suppose $\langle V\setminus\! U\rangle=K_1\cup P_3$, $K_1=\{v_1\}$ ant $P_3=v_2v_3v_4$.
 Then $v_1$ is
adjacent to all the vertices of $U$ and 
 $v_2,v_2$ are adjacent to at least $n-5$ vertices of $U$.
 Hence if $n-5\geqslant 3$, then,  $\td(G)\leqslant 4$ whic is a contradiction.
 Therefore $5\leqslant n\leqslant 7$.
For $n=5$, $G$ is  isomorphic to $C_4(P_2,0,0,0)$ or $C_3(P_2,P_2,0)$ or  graphs $T_7$, $T_8$, $T_9$ as shown in Figures \ref{ft104},\ref{ft105}.
\begin{center}
\begin{tikzpicture}[line width=.4mm,line cap=round,line join=round,>=triangle 45,x=1.0cm,y=1.0cm]
\draw (0.,0.)-- (1.,0.);
\draw (1.,0.)-- (2.,0.);
\draw (3.,0.)-- (1.5,2.);
\draw (1.5,2.)-- (2.,0.);
\draw (1.5,2.)-- (1.,0.);
\draw (1.5,2.)-- (0.,0.);
\draw [fill=black] (0.,0.) circle (2pt);
\draw [fill=black] (1.,0.) circle (2pt);
\draw [fill=black] (2.,0.) circle (2pt);
\draw [fill=black] (3.,0.) circle (2pt);
\draw [fill=black] (1.5,2.) circle (2pt);
\begin{scriptsize}
\draw (1.2,-0.4) node[anchor=north west] {$T_7$};
\end{scriptsize}
\end{tikzpicture}\\
\newfig{ft104}
\end{center}
 For $n=6$ suppose $U=\{x,y\}$.
Then, $G$ is isomorphic to  
one of graphs $T_8$, $T_9$ as shown in Figure \ref{ft105}.
\begin{center}
\begin{tikzpicture}[line width=.4mm,line cap=round,line join=round,>=triangle 45,x=1.0cm,y=1.0cm]
\draw (0.,0.)-- (1.,0.);
\draw (1.,0.)-- (2.,0.);
\draw (0.5,1.)-- (1.,2.);
\draw (1.,2.)-- (1.5,1.);
\draw (1.5,1.)-- (2.,0.);
\draw (2.,0.)-- (0.5,1.);
\draw (0.5,1.)-- (0.,0.);
\draw (3.5,0.)-- (4.5,0.);
\draw (4.5,0.)-- (5.5,0.);
\draw (5.5,0.)-- (5.,1.);
\draw (5.,1.)-- (4.5,2.);
\draw (4.5,2.)-- (4.,1.);
\draw (4.,1.)-- (4.5,0.);
\draw (4.,1.)-- (3.5,0.);
%%%%%%%%%%
%%%%%%%%%%%%%%%%%%%%%
\draw [fill=black] (0.,0.) circle (2pt);
\draw [fill=black] (1.,0.) circle (2pt);
\draw [fill=black] (2.,0.) circle (2pt);
\draw [fill=black] (0.5,1.) circle (2pt);
\draw [fill=black] (1.,2.) circle (2pt);
\draw [fill=black] (1.5,1.) circle (2pt);
\draw [fill=black] (3.5,0.) circle (2pt);
\draw [fill=black] (4.5,0.) circle (2pt);
\draw [fill=black] (5.5,0.) circle (2pt);
\draw [fill=black] (5.,1.) circle (2pt);
\draw [fill=black] (4.5,2.) circle (2pt);
\draw [fill=black] (4.,1.) circle (2pt);
\begin{scriptsize}
\draw (1.15,2.2) node[anchor=north west] {$v_1$};
\draw (4.65,2.2) node[anchor=north west] {$v_1$};
\draw (0,1.2) node[anchor=north west] {$x$};
\draw (3.5,1.2) node[anchor=north west] {$x$};
\draw (1.6,1.2) node[anchor=north west] {$y$};
\draw (5.1,1.2) node[anchor=north west] {$y$};
\draw (0.7,-0.2) node[anchor=north west] {$T_8$};
\draw (4.2,-0.2) node[anchor=north west] {$T_9$};
\end{scriptsize}
\end{tikzpicture}
\\
\newfig{ft105}
\end{center}
 %%%%%%%%%%%%%%%%%%%%%%%%%%%%%%%
For  $n=7$, we have  $\delta(G)=3$. Hence, $V\setminus\! \{v_1,v_2\}$ is a three dominating set which is a contradiction.
  Finally suppose
$\langle V\setminus\! U\rangle=K_2\cup K_2$. 
Then at most three vertices $x,y,z$ in $V$ exsist so that $N(x)\cap (V\setminus\! U)<3$,
$N(y)\cap (V\setminus\! U)<3$ and $N(z)\cap (V\setminus\! U)<3$.
 Hence, $\{v_1,v_2,v_3,v_4,x,y,z\}$ is a three dominating set of $G$.
 Therefore  $n-1=\td(G)\le 7$. This gives $4=|V\setminus\! U|\leqslant n\leqslant 8$.
$n=4$ is impossible.
 For $n=5$, $G$ is isomprphic to $C_3(P_3,0,0)$ or $F_2$.
 For $n=6$, $G$ is isomorphic to  
one of graphs $T_{10}$, $T_{11}$ or $T_{12}$ as shown in figure\ref{ft105}.
\begin{center}
\begin{tikzpicture}[line width=.4mm, line cap=round,line join=round,>=triangle 45,x=1.23cm,y=1.5cm]
\draw (0.4,0.)-- (1.,0.);
\draw (-0.8,0.)-- (-0.2,0.8);
\draw (-0.8,0.)-- (-0.2,0.);
\draw (1.8,0.)-- (2.4,0.);
\draw (3.,0.)-- (3.6,0.);
\draw (3.6,0.)-- (2.4,0.8);
\begin{scriptsize}
\draw (5,-0.2) node[anchor=north west] {$T_{12}$};
\draw (-.15,-0.2) node[anchor=north west] {$T_{10}$};
\draw (2.45,-0.2) node[anchor=north west] {$T_{11}$};
\end{scriptsize}
\draw (-0.2,0.8)-- (-0.2,0.);
\draw (-0.2,0.8)-- (0.4,0.);
\draw (2.4,0.8)-- (3.,0.);
\draw (2.4,0.8)-- (1.8,0.);
\draw (0.4,0.8)-- (-0.2,0.);
\draw (0.4,0.8)-- (1.,0.);
\draw (2.4,0.8)-- (3.,0.8);
\draw (3.,0.8)-- (2.4,0.);
\draw (3.,0.8)-- (3.,0.);
\draw (4.4,0.)-- (5.,0.);
\draw (5.,0.8)-- (4.4,0.);
\draw (5.,0.8)-- (5.6,0.8);
\draw (5.6,0.8)-- (5.,0.);
\draw (5.,0.8)-- (5.6,0.);
\draw (5.6,0.)-- (6.2,0.);
\draw (5.6,0.8)-- (6.2,0.);
\draw [fill=black] (-0.2,0.) circle (2pt);
\draw [fill=black] (0.4,0.) circle (2pt);
\draw [fill=black] (1.,0.) circle (2pt);
\draw [fill=black] (-0.8,0.) circle (2pt);
\draw [fill=black] (-0.2,0.8) circle (2pt);
\draw [fill=black] (1.8,0.) circle (2pt);
\draw [fill=black] (2.4,0.) circle (2pt);
\draw [fill=black] (3.,0.) circle (2pt);
\draw [fill=black] (3.6,0.) circle (2pt);
\draw [fill=black] (2.4,0.8) circle (2pt);
\draw [fill=black] (0.4,0.8) circle (2pt);
\draw [fill=black] (3.,0.8) circle (2pt);
\draw [fill=black] (4.4,0.) circle (2pt);
\draw [fill=black] (5.,0.) circle (2pt);
\draw [fill=black] (5.,0.8) circle (2pt);
\draw [fill=black] (5.6,0.8) circle (2pt);
\draw [fill=black] (5.6,0.) circle (2pt);
\draw [fill=black] (6.2,0.) circle (2pt);
\end{tikzpicture}\\
\newfig{ft107}
\end{center}
For $n=7,8$ we have $\delta(G)\geqslant3$. Suppose  $K_1=\langle v_1,v_2\rangle$ and   $K_2=\langle v_3,v_4\rangle$.
 Then, $\{v_1,v_3\}\cup U$ is a three dominating set in $G$ which is a contradiction.\\
{\bf Case v.} $\td(G)=n$
 and $\kappa(G)=n-5$.
 Then, $n-5\leqslant \delta (G)\leqslant 2$ and $5\leqslant n\leqslant 7$.
  $n=5$ is impossible.
 For $n=7$, $G$ is isomorphic to $C_7$ and for $n=6$, $G$ is isomorphic to $P_6$.
The converse is obvious.
\end{proof}

\bigskip
\bigskip
{\footnotesize {\bf R. Safakish}\; \\
Department of Mathematics, Faculty of Science,\\ bu-ali sina University,\\ Hamedan,  I. R. Iran \\
{\tt Email: safakish@basu.ac.ir}\\
\bigskip
\bigskip
{\footnotesize {\bf S. Mehry}\; \\
Department of Mathematics, Faculty of Science,\\ bu-ali sina University,\\ Hamedan,  I. R. Iran \\
{\tt Email: sh.mehry@basu.ac.ir}\\


\begin{thebibliography}{}
\bibitem{a1} G.Chartrand and L.Lesniak,Graphs and Digraphs, CRC,(2005).
\bibitem{a2} T.W.Haynes,S.T.Hedetniemi and P.J.Slater, Fundamentals of Domination in Graphs, Marcel Dekkar, Inc., NewYork, (1997).
\bibitem{b2}J. F. Fink, M. S. Jacobson, L. F. Kinch, J. Roberts: On graphs having domination number half their
order, Period. Math. Hungar. 16 (1985), 287-293.
\bibitem{a3} T.W.Haynes, S.T.Hedetniemi and P.J.Slater , Domination in Graphs- Advanced Topics, Marcel Dekker, Inc., NewYork , (1997).
\bibitem{b1} B. Hedman. Another extremal problem for Tur´an graphs. Discrete Math., 65(2):173–176, 1987.
\bibitem{a4} F.Harary and T.W.Haynes, Double Domination in graphs, Ars combin. 55, 201- 213(2000).
\bibitem{a5} J.Paulraj Joseph and S.Arumugam, Domination and connectivity in Graphs, International Journal of management and systems, 8,233- 236(1992).
\end{thebibliography}
\end{document}